\documentclass[a4paper,11pt]{amsart}
\usepackage[utf8]{inputenc}
\usepackage{amsmath, amssymb}
\usepackage{setspace}
\usepackage[left=3cm, right=3cm, top=3cm, bottom=3cm]{geometry}                 

\usepackage{graphicx} 
\usepackage{pgf,tikz}
\usetikzlibrary{arrows}
\usepackage{amssymb}
\usepackage{pdfsync}
\usepackage{mathrsfs}
\usepackage{hyperref} 
\usepackage{verbatim} 
\usepackage{epstopdf}
\usepackage{xargs}
\usepackage{todonotes}
\usepackage[normalem]{ulem}

\newcommandx{\jow}[2][1=]{\todo[linecolor=orange,backgroundcolor=orange!25,bordercolor=orange,#1]{#2}}

\newcommandx{\mateus}[2][1=]{\todo[linecolor=blue,backgroundcolor=blue!25,bordercolor=blue,#1]{#2}}

\def\R{\mathbb{R}}

\def\N{\mathbb{N}}
\def\Z{\mathbb{Z}}

\newcommand{\eps}{\varepsilon}
\newcommand{\mmd}{\mathrm{d}}

\newcommandx{\eq}{\approxeq}

\newtheorem{theorem}{Theorem}

\newtheorem{proposition}[theorem]{Proposition}
\newtheorem{lemma}[theorem]{Lemma}

\newtheorem{conjecture}[theorem]{Conjecture}
\newtheorem{theoo}{Theorem}

\numberwithin{equation}{section}
\numberwithin{theorem}{section}

\begin{document}
\title[Maximal estimates and local well-posedness for generalized ZK]{Maximal function estimates and local well-posedness for the generalized Zakharov--Kuznetsov equation}
\author{Felipe Linares}
\address[F. Linares]{IMPA\\
Instituto Matem\'atica Pura e Aplicada\\
Estrada Dona Castorina 110\\
22460-320, Rio de Janeiro, RJ\\Brazil}
\email{linares@impa.br} 
\author{Jo\~ao P.G. Ramos}
\address[J.P. Ramos]{IMPA\\
Instituto Matem\'atica Pura e Aplicada\\
Estrada Dona Castorina 110\\
22460-320, Rio de Janeiro, RJ\\Brazil}
\email{joaopgramos95@gmail.com} 


\begin{abstract}
We prove a high-dimensional version of the Strichartz estimates for the unitary group associated to the free Zakharov--Kuznetsov equation. As a by--product, 
we deduce maximal estimates which allow us to prove local well-posedness for the generalized Zakharov--Kuznetsov equation in the whole subcritical case whenever 
$d \ge 4, k \ge 4,$ complementing the recent results of Kinoshita \cite{Kinoshita1, Kinoshita2} and Herr--Kinoshita \cite{HK}. Finally, we use some of those 
maximal estimates in order to prove pointwise convergence results for the flow of the generalized Zakharov--Kuznetsov equation in any dimension, in the same spirit of \cite{CLS}.
\end{abstract}

\maketitle

\section{Introduction} In a manuscript published in 1974 \cite{ZakharovKuznetsov}, Zakharov and Kuznetsov deduced the equation 
\begin{equation}\label{ZK} 
\partial_t u + \partial_x \Delta u + u \partial_x u = 0,
\end{equation}
where $u = u(x,y,t)$ is a real-valued function, as means to 
describe propagation of ionic-acoustic waves in a magnetized plasma. In \cite{LannesLinaresSaut}, the authors derive \eqref{ZK} from the Euler--Poisson 
system with magnetic field. Our main focus will remain, however, in the \emph{initial value problem} (IVP) associated to the generalized version of this equations, as follows: 
\begin{equation}\label{ZKIVP}
\begin{cases} 
\partial_t u + \partial_x \Delta u + \partial_x(u^{k+1}) = 0  \hskip10pt\text{on}\hskip5pt \R^d\times \R, \;k\in\Z^{+}, \cr 
u(0) = u_0  \hskip95pt\text{on}\hskip5pt \R^d,\cr
\end{cases}
\end{equation}
where the datum $u_0$ is taken to belong to an adequate Sobolev space $H^s(\R^d).$ This problem has attracted the attention of many authors since the work 
of Faminskii \cite{Faminskii} on the $k=1$ case, who initially showed local well-posedness for $s\ge 1$ in the two-dimensional case. Eversince, many others have contributed in 
the two-dimensional case, among which we mention the work of Linares and Pastor \cite{LPa} where it was proved local-wellposednes for $s>3/4$
by employing smoothing effects. The works of Molinet and Pilod \cite{MolinetPilod} and Gr\"unrock and Herr \cite{GruenrockHerr}, which proved local well-posedness 
by using the Fourier restriction method for $s > \frac{1}{2}.$ 

Still in the $k=1$ case, we also mention the work of Molinet--Pilod \cite{MolinetPilod} and Ribaud--Vento \cite{RV2}, which proved local and global well-posedness in $H^s(\R^3)$ for $s > 1$. It was not until recently,
however, in the works of Kinoshita \cite{Kinoshita1} and Herr--Kinoshita \cite{HK} that well--posedness was obtained in the best possible range for the Picard iteration method: 
$s > -\frac{1}{4}$ if $d=2$ and $s > \frac{d-4}{2}\;$ for $d \ge 3.$ 	

A key instrument in all references above in order to study this equation is estimating its \emph{free solutions}. Indeed, let 
\begin{align}\label{eq linear ZK}
\begin{cases}
\partial_t u + \partial_x \Delta u = 0 & \text{ on } \R^d \times \R; \cr
u(x,0) = u_0(x) & \text{ on } \R^d, \cr
\end{cases}
\end{align}
be the linear problem associated to the equation \eqref{ZK}. We denote the group of operators associated to this problem by $U(t)u_0$ given by 
\[
\widehat{U(t)u_0}(\xi,\eta) = e^{it\xi(\xi^2+|\eta|^2)} \widehat{u_0}(\xi,\eta), \; \; \xi \in \R, \eta \in \R^{d-1}.
\]
We remark that, as in the one-dimensional case of the KdV equation, this group possesses good maximal properties: 
indeed, if $d=2,$ Linares and Pastor \cite{LPa} proved that 
\[
\| U(t) u_0\|_{L^4_x L^{\infty}_{y,T}} \lesssim \|u_0\|_s, \, s > 3/4,\; T \in [0,1].
\]
Later on, Gr\"unrock \cite{Gru2} extended this estimate to arbitrary times. Also, Faminskii \cite{Faminskii} showed that the same estimate holds with an $L^2_x$ norm: 
\[
\| U(t) u_0\|_{L^2_x L^{\infty}_{y,t}} \lesssim \|u_0\|_s, \, s > 3/4.
\]
If $d=3,$ Gr\"unrock \cite{Gru1} proved that slightly stronger inequality holds in case $p=4.$ Indeed, he proved that 
\[
\| U(t) u_0 \|_{L^4_{x,y} L^{\infty}_t} \lesssim \| u_0\|_s, \, s>3/4.
\]
This implies, by Sobolev embedding, a maximal estimate of the form 
\[
\| U(t) u_0\|_{L^4_x L^{\infty}_{y,t}} \lesssim \|u_0\|_s, 
\]
for $s>5/4.$ Finally, complementing those results, we have the following \emph{time-weighted} maximal estimate,
proved by Ribaud-Vento \cite{RV2}: 
\[
\|t^{\alpha} U(t) u_0 \|_{L^2_x L^{\infty}_{y,t}} \lesssim \| u_0 \|_s, \, s>1,
\]
where we are allowed to take $\alpha \ge 3/8.$ 

In this note, we would like to address two main questions and assess their consequences, namely: 
\begin{enumerate}
 \item[(i)](Space-time maximal estimate) Given $p \in (1,\infty),$ how large can $s \in \R$ be so that the estimate 
\begin{equation}\label{eq maximal space time} 
\| U(t)u_0 \|_{L^p_x L^{\infty}_{y,t \in [0,1]}} \lesssim \|u_0\|_s 
\end{equation}
holds for all $u_0 \in \mathcal{S}(\R^d)$? 
 \item[(ii)](Time-only maximal estimate) Given $p \in (1,\infty),$ how large can $s \in \R$ be so that the estimate 
\begin{equation}\label{eq maximal time only}
\| U(t)u_0\|_{L^p_{x,y} L^{\infty}_{t \in [0,1]}} \lesssim \|u_0\|_s
\end{equation}
holds for all $u_0 \in \mathcal{S}(\R^d)?$
\end{enumerate}

Regarding those questions, we construct a simple set of counterexamples to give preliminary restrictions on such values of $s.$ Indeed, the partial restrictions we 
have are the following: 
\begin{proposition}\label{thm necessary} Regarding the questions before, the following assertions hold: 
\begin{enumerate}
 \item If \eqref{eq maximal space time} holds, then 
 $$ p \in [2,+\infty) \text{ and }s \ge \frac{d}{2} - \frac{1}{p}.$$ 
 \item If \eqref{eq maximal time only} holds, then 
 $$ p \in [2,+\infty) \text{ and }s \ge \max\left\{d \left(\frac{1}{2} - \frac{1}{p}\right),\frac{3}{2p} - \frac{d}{2}\left(\frac{1}{2} - \frac{1}{p}\right)\right\}.$$ 
\end{enumerate}
\end{proposition}

As a direct consequence, we are able to establish that some previous results in the literature are, in fact, \emph{sharp}. 

\begin{theorem}\label{thm sharp} Consider the group $\{U(t)\}_{t \in \R}$ the unitary group associated to \eqref{eq linear ZK}. 
\begin{enumerate}
 \item If $d=3,$ then the estimates 
 \[
 \| U(t) u_0\|_{L^p_x L^{\infty}_{y,t\in[0,1]}} \lesssim \| u_0 \|_s
 \]
 hold whenever $s > \frac{3}{2} - \frac{1}{p}$ and $p \ge 2.$ Moreover, this estimate is sharp, up to the endpoint $s = \frac{3}{2} - \frac{1}{p}.$ 
 \item If $d=2,$ then the estimates 
 \[
 \| U(t) u_0\|_{L^p_x L^{\infty}_{y,t\in[0,1]}} \lesssim \|u_0 \|_s
 \]
 hold whenever either $s > 1 - \frac{1}{p}$ and $p \ge 4$ \emph{or} $p \in [2,4)$ and $s > \frac{3}{4}$. Moreover, this estimate is sharp, up to the endpoint
 $s = 1 - \frac{1}{p},$ in case $p \in [4,+\infty)$, or up to the endpoint $s=\frac{3}{4}$ in case $p=2.$ 
\end{enumerate}
\end{theorem}

In order to prove this theorem, we will use the already existing estimates for the group, which have been noted before (e.g. \cite{Faminskii,RV2,Gru1,Gru2} and others). 
The most interesting part of the proof of the result is the sharpness of the estimates when $d=2,p=2,$ for which we actually resort to PDE methods. 

Indeed, we show that, in case the estimate were true for some $s < \frac{3}{4}$ and $d=2,p=2,$ the Picard iteration method would work for proving local 
well-posedness for the modified Zakharov--Kuznetsov equation in dimension 2 for some $s_0 < \frac{1}{4}$. By a recent result of Kinoshita \cite{Kinoshita2}, however, we know that, similarly 
as in the KdV case, we can only prove local well-posedness by a contraction argument in case $s \ge \frac{1}{4}.$ This contradiction establishes our result

These methods provide us, through Sobolev embedding and an argument by contradiction, with the $s > d \left(\frac{1}{2} - \frac{1}{p}\right)$ range for the second part of 
Proposition \ref{thm necessary}. In order to obtain the range mentioned in that statement, which does better in case $p \le \frac{2(d+1)}{d},$ we employ results by Sj\"olin \cite{Sj1} and Rogers \cite{Rogers} on 
maximal estimates related to Schr\"odinger-like operators. See Section 2 below for more details. 

Next, we establish \emph{Strichartz estimates} for the unitary group $U(t)$ in higher dimensions. We remark that many of those estimates were obtained in the two and three dimensional cases in \cite{LPa}, \cite{LSaut}. In the 
current format, these follow closely \cite{HLRRW}, and are contained in the recent work of Schippa \cite{Schippa} for the critical line $\frac{2}{q} + \frac{2}{r} = 1.$ 

\begin{proposition}\label{thm strichartz} Let $d \ge 3,$ and $q,r \ge 2.$ It holds that 
\[
\| U(t) u_0\|_{L^r_{t \in [0,1]} L^q(\R^d)} \lesssim \| u_0\|_{H^s},
\]
where $\frac{2}{q} + \frac{2}{r} \le 1,$ and $s= d\left( \frac{1}{2} - \frac{1}{q}\right) - \frac{3}{r}.$ 
\end{proposition}

In particular, setting $q=r=4$ in this result implies a recent estimate by Herr--Kinoshita \cite{HK} (see also \cite{Kinoshita2}), which enables the authors to prove subcritical local well-posedness for the 
Zakharov--Kuznetsov equation in any dimension:
\[
\| U(t) u_0 \|_{L^4_{x,y,[0,1]}} \lesssim \| \langle \nabla \rangle^{\frac{d-3}{4}} u_0 \|_{L^2},
\]
whenever $d \ge 3.$ An application of Sobolev embedding in the $t-$variable readily implies the following \emph{maximal estimate:} 
\[
\| U(t) u_0 \|_{L^4_{x,y}L^{\infty}_{[0,1]}} \lesssim \| u_0\|_{H^s}, 
\]
whenever $s > \frac{d}{4}.$  Notice that this estimate is, in fact, \emph{sharp} according to Proposition \ref{thm necessary}, as so is the other \emph{space-time maximal estimate} which follows directly from the one above
via Sobolev embedding, this time on the $y-$variable: 
\begin{equation}\label{eq L4d}
\| U(t) u_0 \|_{L^4_x L^{\infty}_{y,[0,1]}} \lesssim \|u_0\|_{H^s},
\end{equation}
for $s > \frac{d}{2} - \frac{1}{4}.$ 

Our final contribution is an usage of maximal estimates in order to prove \emph{well-posedness} for the generalized Zakharov--Kuznetsov equation in higher dimensions:
\begin{equation}\label{eq gZK}
\begin{cases}
\partial_t u + \partial_x \Delta u + \partial_x (u^{k+1}) = 0 \hskip15pt \text{ on } \R^d \times \R; \cr 
 u(x,0) = u_0(x) \hskip75pt \text{ on } \R^d.  
\end{cases}
\end{equation}

The modified and generalized versions of the Zakharov--Kuznetsov equation have been previously considered by many authors. In dimension 2, we mention briefly the works of Biagioni and Linares \cite{BL}, Linares and 
Pastor \cite{LPa,LPa2}, and Farah, Linares and Pastor \cite{FLPa}, which all contributed to the development of the topic. In particular, we mention the work of Ribaud and Vento \cite{RV1} which proved local well-posedness 
for \eqref{eq gZK} whenever $k \ge 4$ in the full subcritical range $ s > 1 - \frac{2}{k},$ and for $s > \frac{5}{12}$ for $k=3,$ $s > \frac{1}{4}$ for $k=2.$ 

It was not until the works of Gr\"unrock \cite{Gru1, Gru2} that local well-posedness was proved in the full subcritical range $s > \frac{1}{3}$ for the $k=3$ case in $d=2,$ and also 
the full subcritical range $s > \frac{3}{2} - \frac{2}{k}$ was reached for any $k \in \Z, \, k \ge 2$ in $d=3.$ Finally, we remark that the subcritical result $s>0$ for $k=2$ was recently shown by Kinoshita 
\cite{Kinoshita1} not to hold if one demands \emph{smoothness} of the flow map, which demonstrates, among other things, that the Ribaud--Vento result for $s>\frac{1}{4}$ in \cite{RV2} was essentially sharp. 
Kinoshita also proved local well-posedness for small data in $s=\frac{1}{4}$ and full subcritical well-posedness $s > \frac{d}{2} - 1$ in any dimension $d \ge 3$ for the modified Zakharov--Kuznetsov equation.
He also shows small data global well-posedness for the critical space $H^{d/2-1}$ for all $d\ge 3.$

Our result, as previously mentioned, complements the results of Herr--Kinoshita \cite{HK}, Kinoshita \cite{Kinoshita2} and Gr\"unrock \cite{Gru1,Gru2}. 

\begin{theorem}\label{thm gZK dk} Let $d \ge 3$ and $k \ge 4.$ Then there are function spaces $\mathcal{X}^s_T$ so that for each $u_0 \in H^s(\R^d)$ with $s > \frac{d}{2} - \frac{2}{k},$ 
the IVP \eqref{eq gZK} has a unique solution 
\[
u \in C([0,T] \colon H^s(\R^d) \cap \mathcal{X}^s_T,
\]
where $T = T(\|u_0\|_s) > 0.$ Moreover, the map $u_0 \mapsto u$ from $H^s(\R^d)$ to $\mathcal{X}^s_T \cap C([0,T] \colon H^s(\R^d))$ is locally Lipschitz continuous. 
\end{theorem}

In order to prove this result, we will use \eqref{eq L4d} in conjunction with a local smoothing estimate. This is heavily inspired in the works of Ribaud and Vento \cite{RV1, RV2} and Gr\"unrock \cite{Gru1, Gru2}, and 
we only miss the full subcritical range in the $k=3$ case because of the lack of a sharp $L^2_x L^{\infty}_{y,T}$ estimate for $U(t);$ see the comments in the last section for more details. 

Finally, we use the same $H^{\frac{d-3}{4}^+} \to L^4_{x,y,T}$ estimate in order to prove a result about \emph{pointwise convergence} of the flow to the initial data. 
\begin{theorem}\label{thm pointwise} 
Let $d, k \in \N$ as before. Let also $k \ge 2$ if $d \in \{2,3\}$ and $k \ge 4$ if $d \ge 4,$ and $s > \max\left(\tilde{s}_d, \frac{d}{2} - \frac{2}{k}\right),$ where 
\[
\tilde{s}_d = \begin{cases} 
              \frac{1}{2} & \text{ if } d =2,3; \\
              \frac{d}{4} & \text{ if } d\ge 4. \
              \end{cases}
\]
Then, for each initial datum $u_0 \in H^s(\R^d),$ the unique solution $u \in C([0,T] \colon H^s(\R^d))$ to the IVP \eqref{eq gZK} given by Theorem \ref{thm gZK dk} converges \emph{pointwise} to $u_0;$ i.e.,
\[
\lim_{t \to 0} u(x,t) = u_0(x), \, \text{ for a.e. } x \in \R^d.
\]
\end{theorem}

Increasing attention has been given to such kinds of results in the recent literature, especially in the KdV case, where we mention the works of Erdo\u{g}an and Tzirakis both in the KdV and Schr\"odinger cases \cite{ET1, ET2}, as well as the very recent 
results of Compaan, Luc\`a and Staffilani \cite{CLS}, which were responsible for establishing such results also in the context of the nonlinear Schr\"odinger equation in in higher dimensions in a sharp range, using the sharp 
pointwise convergence theorem for the Schr\"odinger flow \cite{DGL, DZ}. In our companion paper \cite{LinaresRamos}, we establish such a pointwise convergence result also in the case of the 
$L^2-$critical generalized Zakharov--Kuznetsov equation in three dimensions, i.e., $k=\frac{4}{3}$ in \eqref{eq gZK}. There, our methods give us as range for well-posedness $s > \frac{3}{4},$ which coincides with 
the range in which the $H^s \to L^4_{x,y} L^{\infty}_T$ maximal estimates hold. 

\subsection*{Notation and Organization} We use the modified Vinogradov equation $A \lesssim B$ several times to indicate that there is an absolute constant $C>0$ so that $A \le C \cdot  B.$ We also use the original Vinogradov equation
$A \ll B$ to denote that there is a (relatively) large constant $C$ with the property $A \le C \cdot B.$ We also use several times the notation $(-\Delta)^{s/2} f = \langle \nabla \rangle^s f = \mathcal{F}^{-1} ( (1+|\xi|^2+|\eta|^2)^{s/2} \widehat{f}).$ 

Finally, the paper is organized as follows. In Section 2, we discuss the counterxamples and the proof of Proposition \ref{thm necessary} and Theorem \ref{thm sharp}. In Section 3, we prove Proposition \ref{thm strichartz}, followed 
by a discussion on how the results in those sections imply well-posedness for the generalized Zakharov--Kuznetsov in higher dimensions and a proof of Theorem \ref{thm pointwise}. 
Finally, in Section 4, we discuss some generalizations, remarks and open questions which arise naturally from our discussion.

\section{Proof of Theorem \ref{thm sharp}}\label{sec sharp}

In this section, we will prove Theorem \ref{thm sharp} by first showing, via an elementary counterexample, Proposition \ref{thm necessary}. We will then show that the necessary conditions provided
by that proposition are, in fact, \emph{sharp} in certain instances when $d=2,3.$ The proof of such sharpness assumes previous results in the literature, and, as stated earlier, 
also an argument involving the modified Zakharov--Kuznetsov equation in two dimensions. 

\subsection{Proof of Proposition \ref{thm necessary}} 
As mentioned in the introduction, we construct a simple set of counterexamples to prove Proposition \ref{thm necessary}.
First, define the Schwartz functions $\varphi_{j,k}$ on the Fourier side as 
\[
 \widehat{\varphi_{j,k}}(\xi,\eta) = \theta(2^j \xi) \psi (2^{-k}\eta),
\]
where $\theta$ is a smooth one-dimensional function supported at $[-4,-1/2] \cup [1/2,4]$, equal to $1$ on $[-2,-1]\cup[1,2],$ and $\psi$ is a $(d-1)-$dimensional 
counterpart of $\theta.$ If $j \ge -k,$ a simple calculation implies that 
\[
\| \varphi_{j,k} \|_s^2 \sim 2^{(d-1)k-j} 2^{2ks}.
\]
On the other hand, for $|y| \ll 2^{-k}, \, |t| \ll \min(1,2^{j-2k}),$ and $|x| \ll 2^j,$ the Fourier transform definition of $U(t)$ implies that 
\[
|U(t) \varphi_{j,k} (x,y)| \gtrsim 2^{(d-1)k-j}. 
\]
This holds as the smallness assumptions are basically cancelling off the phase in the inverse Fourier transform definition of $U(t)\varphi_{j,k}.$ Thus, 
\[
2^{(d-1)k-j+\frac{j}{p}} \lesssim \| U(t)\varphi_{j,k} \|_{L^p_x L^{\infty}_{y,t}} \lesssim \| \varphi_{j,k} \|_s = 2^{((d-1)k-j)/2} 2^{ks}. 
\]
In other words, $2^{\frac{d-1}{2}k +j(1/p - 1/2)} \lesssim 2^{ks}.$ Here we have set ourselves the freedom to choose $j \ge -k.$ If $p<2,$ we can just let $j \to \infty$ 
to show that, in fact, no inequality of the type \eqref{eq maximal space time} can hold. For $p \ge 2,$ the worst case scenario happens when $j = -k,$ so that 
\[
2^{\frac{d-1}{2}k +k(1/2 - 1/p)} \lesssim 2^{ks}, \forall k \ge 1 \iff s \ge \frac{d}{2} - \frac{1}{p}.
\]
This proves the first part of \ref{thm necessary}. In order to prove the second one, assume that the inequality 
\[
\| \sup_t |U(t)u_0| \|_p \lesssim \|u_0\|_r
\]
holds with $p>1.$ By Sobolev embedding, this estimate implies that \eqref{eq maximal space time} holds for all $s > r+\frac{d-1}{p}.$ Therefore, 
\[
r \ge d \left( \frac{1}{2} - \frac{1}{p}\right). 
\]
This finishes the proof of the first lower bound on the second part. In order to prove the $s> \frac{3}{2p} - \frac{n}{2} \left(\frac{1}{2} - \frac{1}{p}\right)$ restriction, we shall use 
two different results. The first of them is the following bound obtained by Sj\"olin in \cite[Theorem~1]{Sj1}.

\begin{theoo}\label{theoo a} Suppose that $\Omega:\R^d \to \R$ is a smooth homogeneous polynomial of degree $\ge 1.$ Then the estimate 
\[
\left\| \sup_{t \in [0,1]} |e^{it\Omega(D)} f| \right\|_{L^p(B^d(0,1))} \lesssim \|f\|_{H^s(\R^d)}
\]
can only hold if $s + \frac{n-1}{2p} \ge \frac{n}{4}.$ Here, we let $\,\widehat{\Omega(D)f}(\zeta) = \Omega(\zeta) \widehat{f}(\zeta).$ 
\end{theoo}
As an immediate corollary, we see that the bound 
\[
\| \sup_{t \in [0,1]} |U(t)f| \|_{L^p(B^d(0,1))} \lesssim \|f\|_{H^s(\R^d)}
\]
can only hold if $s + \frac{n-1}{2p} \ge \frac{n}{4}.$ In order to pass from the unit ball to the whole euclidean space, we must use the following result by Rogers \cite[Theorem~1.3]{Rogers}:

\begin{theoo}\label{theoo b} Let $p \ge 2$ and $\Omega:\R^d \to \R$ be a smooth function so that, for some integer $m \ge 2,$ it holds that $|D^{\alpha} \Omega(\zeta)| \lesssim_{\alpha} |\zeta|^{m - |\alpha|}$ for all
multiindices $|\alpha| \le 2,$ and that $|\nabla \Omega (\zeta)| \gtrsim |\zeta|^{m-1}.$ Then the \emph{local} maximal estimate 
\[
\left\| \sup_{t \in [0,1]} |e^{it\Omega(D)} f| \right\|_{L^p(B^d(0,1))} \lesssim \|f\|_{H^s(\R^d)}
\]
holds for all $s > s_0$ if and only if the \emph{global} maximal estimate
\[
\left\| \sup_{t \in [0,1]} |e^{it\Omega(D)} f| \right\|_{L^p(\R^d)} \lesssim \|f\|_{H^s(\R^d)}
\]
holds for all $s > m\,s_0 - (m-1)d \left(\frac{1}{2} - \frac{1}{p}\right).$ 
\end{theoo}

In our case, it is easy to see that $\Omega(\xi,\eta) = \xi(\xi^2 + |\eta|^2)$ satisfies the hypotheses above with $m = 3.$ Therefore, Theorems \ref{theoo a} and \ref{theoo b} imply together that the 
time-only maximal estimate \ref{eq maximal time only} can only hold whenever 
$ s > 3 \left(\frac{d}{4} - \frac{d-1}{2p}\right) - 2d \left(\frac{1}{2} - \frac{1}{p}\right) = \frac{3}{2p} - \frac{d}{2} \left(\frac{1}{2} - \frac{1}{p}\right).$ This 
completes our proof.

\subsection{Analysis of the three-dimensional maximal estimate} In three-dimensions, the existing results are (essentially) sharp and completely compatible with the counterexamples
given above. Indeed, it is not complicated to prove that, for $p \to \infty,$ Faminskii's methods imply that 
\begin{equation}\label{eq almost linfty}
\| U(t) u_0 \|_{L^{\infty-}_{x,y,[0,1]}} \lesssim \| u_0 \|_{s}, \, s > 3/2.
\end{equation}
Indeed, \eqref{eq almost linfty} follows by a simple Sobolev embedding argument: 
\[
\| U(t) u_0 \|_{L^N_{x,y,t \in [0,1]}} \lesssim \| \langle \nabla_{x,y} \rangle^{\frac{3(N-2)}{2N}} U(t) u_0\|_{L^N_{[0,1]} L^2_{x,y}} \le \|u_0\|_{H^{\frac{3(N-2)}{2N}}}.
\]
We are then allowed to directly interpolate this bound with Ribaud--Vento \cite[Proposition~3.3]{RV2}, which states
\[
\| U(t) u_0 \|_{L^2_x L^{\infty}_{y,[0,1]}} \lesssim \|u_0\|_{H^s}, \, s > 1.
\]
This yields the result directly by letting $N \to \infty.$ 

Alternatively, one can also follow Faminskii's approach. We are then required to estimate functions of the form 
\[
I_k(x,y,t) = \int_{\R^3} e^{i(x\xi + y \cdot\eta + t \xi(\xi^2+|\eta|^2))} \phi(2^{-k}(\xi,\eta)) \, \mmd \xi \, \mmd \eta,
\]
where $\phi$ is a smooth indicator function of the annulus $\{|(\xi,\eta)| \sim 1\}.$ We clearly have $\|I_k\|_{L^{\infty}_{x,y,t}} \le 2^{3k}.$ By \cite[Lemma~3.3]{RV2},
we have $\|I_k\|_{L^1_x L^{\infty}_{y,t}} \lesssim \text{poly}(k) 2^{2k}.$ Interpolation yields 
\[
\| I_k \|_{L^p_x L^{\infty}_{y,t}} \lesssim \text{poly}(k) 2^{2k + (1-\theta)k},
\]
where $p = \frac{1}{\theta}.$ An application of the $TT^*$ method and a Littlewood--Paley analysis yields 
\[
\| U(t)u_0\|_{L^q_x L^{\infty}_{y,t}} \lesssim \|u_0\|_s, \, s > 1 + \frac{1-\theta}{2}, \, \text{ for } q = \frac{2}{\theta}.
\]
On the other hand, inequality \eqref{eq maximal space time} only holds if $s \ge \frac{3}{2} - \frac{1}{p}.$ Setting $p=q=\frac{2}{\theta}$ gives us that the result is, indeed, sharp, up to
the endpoint, as we wished. 

\subsection{Analysis of the two-dimensional maximal estimate}\label{sec 2d counter} For the case of $d=2,$ we have a distinction to make: 

\noindent If $p \ge 4,$ then the counterexample we gave yields the sharp bound, as 
Linares--Pastor and Gr\"unrock \cite{LPa,Gru1} show that 
\[
\| U(t) u_0 \|_{L^4_x L^{\infty}_{y,t}} \lesssim \| u_0\|_s, \, s > 3/4. 
\]
By the Sobolev embedding argument from the previous section, one easily sees that the range of $s$ one gets matches with that of our counterexample if $p \ge 4,$ except for the endpoint $s = 1 - \frac{1}{p}.$ 

On the other hand, if $p \in [2,4),$ the subject is much more delicate. Indeed, we will show that Faminskii's result
\begin{equation}\label{eq faminskii}
\| U(t)u_0 \|_{L^2_x L^{\infty}_{y,t}} \lesssim \|u_0\|_s, 
\end{equation}
for $s > 3/4,$ is optimal, up to the endpoint $s=3/4.$ Instead of building an explicit counterexample, we argue by contradiction. 

In fact, suppose that \eqref{eq faminskii} holds for $s > s_0, \, 3/4 > s_0.$ We follow Ribaud--Vento's ideas for the well-posedness of the modified Zakharov--Kuznetsov equation \cite{RV1}. More precisely,
we wish to construct a solution to 
\begin{equation}\label{eq mZK}
\begin{cases} 
 \partial_t u + \partial_x \Delta u + \partial_x(u^3) = 0 & \text{ on } \R \times \R^2, \cr 
 u(x,0) = u_0(x) & \text{ on } \R^2 \cr
\end{cases}
\end{equation}
for all $u_0 \in H^s,$ with $s$ to be specified in a while. The idea is that, if $s$ is sufficiently small and we are able to prove that the data-to-solution map $u_0 \mapsto u$ above is, in fact, 
\emph{smooth,} then we will directly contradict Kinoshita's result, which shows that the data-to-solution mapping cannot possess $C^3$ regularity in the case of the modified Zakharov--Kuznetsov equation in 
dimension 2. 

To that extent, define the auxiliary norms
\[
\| u \|_{Y^s_T} = \|u\|_{L^{\infty}_T H^s_{x,y}} + \| \langle \nabla_{x,y} \rangle ^{s+1} u \|_{L^{\infty}_x L^2_{y,T}} + \| \langle \nabla_{x,y} \rangle ^{s - s_0^+} u \|_{L^2_x L^{\infty}_{y,T}},
\]
and, after that, the space in which we wish to construct a solution as $X^s_T = \{ u \in C([0,T];H^s(\R^d)), \|u\|_{X^s_T} < \infty\},$ where 
\[
\| u \|_{X^s_T} = \| \|\Delta_k u \|_{Y^s_T} \|_{\ell^2(\N)}.
\]
Here $\mathcal{F}_{x,y}(\Delta_k u)(\xi,\eta,t) = \phi(2^{-k} (\xi,\eta)) \mathcal{F}_{x,y}{u}(\xi,\eta,t)$ denotes the $k-$th Littlewood--Paley projection in frequency. Because $U(t)$ is unitary, 
the presence of Kato smoothing and the fact that we supposed that \eqref{eq faminskii} holds for $s>s_0,$ then 
\begin{equation}\label{eq linear}
\| U(t)u_0 \|_{X^s_T} \lesssim \|u_0\|_s, \; \forall s \in \R. 
\end{equation}
Following the Duhamel formulation of \eqref{eq mZK}, we wish to prove that the map 
\begin{equation}\label{eq duhamel}
F(u) = U(t)u_0 + \int_0^t U(t-t')\partial_x(u^3) \, \mmd t' 
\end{equation}
has a fixed point in the space $E(T,a)$ of functions whose $X^s_T-$norm is at most $a.$ The linear part is controlled because of \eqref{eq linear}, so we focus on the non-linear part. By the methods from
both \cite{RV1} and \cite{RV2}, the non-linear part of \eqref{eq duhamel} can be estimated by 
\[
\left\| \int_0^t U(t-t') \partial_x (u^3) \, \mmd t' \right\|_{X^s_T} \lesssim \| 2^{sk} \| \Delta_k(u^3) \|_{L^1_xL^2_{y,T}} \|_{\ell^2(\N)}. 
\]
Our goal is to estimate the right-hand side in terms of a power of the $X^s_T$ norm of $u.$ To that extent, we notice that the definition of $Y^s_T$ gives us 
\[
2^{(s+1)k} \| \Delta_k u \|_{L^{\infty}_x L^2_{y,T}} \lesssim \| \Delta_k u \|_{Y^s_T}  
\]
and 
\[
2^{(s-s_0^+)k} \| \Delta_k u \|_{L^2_x L^{\infty}_{y,T}} \lesssim \| \Delta_k u \|_{Y^s_T}. 
\]
Interpolation yields then 
\begin{equation}\label{eq interpolation}
2^{\alpha k} \| \Delta_k u \|_{L^p_x L^q_{y,T}} \lesssim \| \Delta_k u \|_{Y^s_T},
\end{equation}
whenever $\frac{1}{p} = \frac{1-\theta}{2}, \, \frac{1}{q} = \frac{\theta}{2},$ and $\alpha = (s +(1+s_0)\theta - s_0)^-.$ In particular, taking $\theta = \frac{s_0}{1+s_0}^+,$ we get 
\[
2^{sk} \| \Delta_k u \|_{L^{p_0^+}_x L^{q_0^-}_{y,T}} \lesssim \|\Delta_k u\|_{Y^s_T},
\]
with $q_0 = \frac{2(1+s_0)}{s_0}, \, p_0 = 2(1+s_0).$ 
Employing the paraproduct decomposition 
\[
\Delta_k(u^3) = \Delta_k \left[ (P_0 u)^3 + \sum_{j \ge k} (\Delta_{j+1} u) \left((P_{j+1} u)^2 + (P_j u)(P_{j+1} u) + (P_j u)^2\right) \right]
\]
together with H\"older's inequality several times shows that 
\begin{equation}\label{eq paraproduct bound}
\| \Delta_k (u^3) \|_{L^1_xL^2_{y,T}} \lesssim \sum_{l \ge k} \| \Delta_l u \|_{L^{p_0^+}_x L^{q_0^-}_{y,T}} \|P_l u\|_{L^{2p_1^-}_x L^{2q_1^+}_{y,T}}^2 + \|P_0 u\|_{L^{p_0^+}_x L^{q_0^-}_{y,T}}  \|P_0 u\|_{L^{2p_1^-}_x L^{2q_1^+}_{y,T}}^2
\end{equation}
where $p_1 = \frac{p_0}{p_0 - 1}, \, q_1 = \frac{2q_0}{q_0 - 2}.$ Here we denote by $P_j u$ the smooth (space) frequency projection of $u$ onto the ball of center zero and radius $\sim 2^j$ given by $P_j = \sum_{l \le j} (\Delta_l u)$. We then take 
$\theta = \frac{1}{q_1}^-$ on \eqref{eq interpolation}. Notice that $q_1 = \frac{4(1+s_0)/s_0}{2/s_0} = 2(1+s_0) = p_0.$ This implies that 
\[
2^{(s+\frac{1}{2} - s_0)^-k} \| \Delta_k u \|_{L^{2p_1^-}_xL^{2q_1^+}_{y,T}} \lesssim \|\Delta_k u\|_{Y^s_T}.
\]
We can therefore estimate 
\[
\| P_j u \|_{L^{2p_1^-}_x L^{2q_1^+}_{y,T}} \lesssim \sum_{k \le j}  \| \Delta_k u \|_{L^{2p_1^-}_xL^{2q_1^+}_{y,T}} \lesssim \left(\sum_{k \le j} 2^{-(s+\frac{1}{2} - s_0)^-k} \right) \|u\|_{X^s_T}.
\]
The sum above converges as long as $s > s_0 - \frac{1}{2}.$ Therefore, 
\[
\| P_j u \|_{L^{2p_1^-}_x L^{2q_1^+}_{y,T}} \lesssim \|u\|_{X^s_T}.
\]
In order to finish the iteration argument, we note the simple estimate 
\begin{equation}\label{eq simple holder}
2^{(s-1)^+j} \| \Delta_j u \|_{L^N_{x,y,T}} \lesssim T^{\delta_N} \| u \|_{X^s_T}, 
\end{equation}
which holds by Sobolev embedding and H\"older's inequality, for $N \gg 1.$ Using the bound on the right hand side of \eqref{eq paraproduct bound} and interpolating with \eqref{eq simple holder} for $N$ arbitrarily large, we get that 
\[
\| 2^{sk} \| \Delta_k (u^3) \|_{L^1_xL^2_{y,T}} \|_{\ell^2(\N)} \lesssim T^{\delta} \|u\|_{X^s_T}^2 \| (1_{j\ge 0} 2^{-sj}) * \|(\Delta_j u)\|_{Y^s_T}\|_{\ell^2(\N)}.
\]
By the discrete version of Young's convolution inequality, the last expression is bounded by $T^{\delta} \|u\|_{X^s_T}^3,$ with $\delta > 0.$ 
Notice that we can run this argument whenever $s>s_0 - 1/2.$ Therefore, by Picard's iteration method, we 
would obtain the local well-posedness of \eqref{eq mZK} in $H^s, \, s>s_0-1/2.$ Moreover, we actually show that, if $s > s_0 - 1/2,$ there is $\delta >0$ so that 
\[
\| F(u) \|_{X^s_T} \le c\|u_0\|_s + C T^{\delta} \|u\|_{X^s_T}^3.
\]
As the nonlinearity in \eqref{eq mZK} is smooth, a standard technique using the Duhamel formulation shows that the data-to-solution map given by the solution of the IVP \eqref{eq mZK} is, in fact, $C^{\infty}-$\emph{smooth} for all $s > s_0 - 1/2.$ But the recent result by Kinoshita
\cite[Theorem~1.3]{Kinoshita2} proved that the data-to-solution map induced by \eqref{eq mZK} 
is not $C^3-$smooth for $s< 1/4,$ and thus, as $s_0 < 3/4,$ this is a contradiction, which finally implies that $s_0 \ge \frac{3}{4},$ as originally wished. 

\section{Proof of Theorem \ref{thm gZK dk}}

Finally, we prove the local well-posedness result for the generalized Zakharov--Kuznetsov equation in higher dimensions. 

\subsection{Proof of Proposition \ref{thm strichartz}}

In this subsection, we prove some auxiliary Strichartz estimates, which will enable us to prove the local well-posedness results for the generalized Zakharov--Kuznetsov equation in any dimension $d > 3,$ given that 
$k \ge 3.$ 

We start by proving a dispersion estimate for the group $U(t) =: e^{-t\partial_x \Delta}.$ 
\begin{proposition}\label{prop dispersion} Let $d \ge 3$ and $p \in [2,+\infty).$ Then for each $t\in \R$ it holds that 
\[
\|e^{-t\partial_x \Delta} f\|_{L^p} \lesssim |t|^{-2\left(\frac{1}{2} - \frac{1}{p}\right)} \left\| \langle \nabla_{x,y} \rangle^{(d-3) \left(1 - \frac{2}{p}\right)}f\right\|_{L^{p'}},
\]
for each $f \in \mathcal{S}(\R^d).$
\end{proposition}

\begin{proof} We first bound the left-hand side of the conclusion of Proposition \ref{prop dispersion} as
\begin{align}\label{eq LP1}
\|U(t)f\|_{L^p} &\lesssim \left( \sum_{k \ge 0} \|U(t) \Delta_k f\|_{L^p}^2 \right)^{1/2},
\end{align}
so that it suffices to bound each summand on the right-hand side of \eqref{eq LP1}. In order to do so, it suffices to prove that 
\begin{equation}\label{eq single scale}
\|U(t) \Delta_0 f\|_{L^p} \lesssim |t|^{-2\left(\frac{1}{2} - \frac{1}{p}\right)} \|f\|_{L^{p'}}.
\end{equation}
Indeed, suppose \eqref{eq single scale} holds, and let $f(x) = \frac{1}{2^{kd}} (\tilde{\Delta}_k g)(x/2^k)$ above. A simple computation then shows that
$$ U(t)\Delta_0 f (x) = 2^{-kd} U(t/2^{3k}) \Delta_k g (x/2^k).$$ 
Therefore, 
\begin{align*}
& 2^{kd\left(\frac{1}{p}-1\right)} \| U(t/2^{3k}) \Delta_k g \|_{L^p} = \| U(t) \Delta_0 f \|_{L^p} \cr 
\lesssim & |t|^{-2\left(\frac{1}{2} - \frac{1}{p}\right)} \|f\|_{L^{p'}} = |t|^{-2\left(\frac{1}{2}-\frac{1}{p}\right)} 2^{kd\left(\frac{1}{p'}-1\right)} \|\tilde{\Delta}_k g\|_{L^{p'}}, \cr
\end{align*}
for all $t \in \R.$ Rearranging terms, we get that 
\begin{equation*}
\begin{split}
\|U(t')\Delta_k g\|_{L^p} &\lesssim |t'|^{-2\left(\frac{1}{2}-\frac{1}{p}\right)} 2^{k(d-3)\left(1-\frac{2}{p}\right)} \|\tilde{\Delta}_k g\|_{L^{p'}} \\
&\sim |t'|^{-2\left(\frac{1}{2}-\frac{1}{p}\right)} \| \langle \nabla_{x,y} \rangle^{k(d-3)\left(1-\frac{2}{p}\right)} \tilde{\Delta}_k g \|_{L^{p'}},
\end{split}
\end{equation*}
for all $t' \in \R.$ By \eqref{eq LP1}, we have then
\begin{align*}
\|U(t)g\|_{L^p} &\lesssim |t|^{-2\left(\frac{1}{2}- \frac{1}{p}\right)} \left( \sum_{k \ge 0}\| \langle \nabla_{x,y} \rangle^{k(d-3)\left(1-\frac{2}{p}\right)} \tilde{\Delta}_k g \|_{L^{p'}}\right) \cr
		&\lesssim |t|^{-2\left(\frac{1}{2} - \frac{1}{p}\right)} \left\| \left(\sum_{k \ge 0} | \Delta_k \langle \nabla_{x,y} \rangle^{k(d-3)\left(1-\frac{2}{p}\right)} g|^2\right)^{1/2} \right\|_{L^{p'}} \cr
		&\lesssim |t|^{-2\left(\frac{1}{2} - \frac{1}{p}\right)} \|\langle \nabla_{x,y} \rangle^{k(d-3)\left(1-\frac{2}{p}\right)} g\|_{L^{p'}},
\end{align*}
by the Littlewood--Paley theorem. 

In order to prove \eqref{eq single scale}, we prove the endpoints $p=2$ and $p=\infty$ and interpolate. For the $p=2$ case, we simply use Plancherel's theorem to obtain
\[
\|U(t) \Delta_0 f \|_{L^2} \lesssim \|f\|_{L^2}.
\]
For the $p=\infty$ case, we rewrite 
\begin{align*}
U(t)\Delta_0 f(x) & = \mathcal{F}^{-1}(\psi(|\xi|)e^{it\xi_1 |\xi|^2} \widehat{f}(\xi))(x) = \int_{\R^d} \psi(|\xi|)e^{it\xi_1 |\xi|^2} \widehat{f}(\xi) e^{-i x \cdot \xi} \, \mmd \xi\cr
		  & = \int_{\R^d} f(y) \left(\int_{\R^d} \psi(|\xi|) e^{it\xi_1 |\xi|^2 + i(x-y) \cdot \xi}  \, \mmd \xi \right) \, \mmd y,
\end{align*}
by Fourier inversion. Taking the modulus inside and using the following lemma yields the proposition:

\begin{lemma}[Proposition 2.1 in \cite{Schippa}] Let $\psi: \R^d \to \R$ be a smooth radial function supported in $B_d(0,2) \backslash B_d(0,1/2).$ Then it holds that 
\[
\left|\int_{\R^d} \psi(|\xi|) e^{it\xi_1 |\xi|^2 + i z \cdot \xi}  \, \mmd \xi \right| \le C |t|^{-1},
\]
where $C>0$ does not depend on $z \in \R^d.$ 
\end{lemma}
See \cite{Schippa} (and alternatively \cite[Proposition~3.1]{HLRRW}) for a proof. 

As previously remarked, we are now done. Indeed, it then holds that
$$
|U(t) \Delta_0 f(x)| \lesssim |t|^{-1} \|f\|_{L^1},
$$ 
which, by the previous considerations, finishes our proof.
\end{proof}

\begin{proof}[Proof of Proposition \ref{thm strichartz}] Now that we have Proposition \ref{prop dispersion}, we can perform the usual $TT^*$ method. First, consider the $\frac{2}{q} + \frac{2}{r} = 1$ line.
By the usual duality arguments, the conclusion follows if we prove that
\[
\left\| \int (-\Delta)^s U(t-t') F(\cdot,t') \, \mmd t' \right\|_{L^r_{t \in [0,1]} L^q(\R^d)} \lesssim \|F\|_{L^{r'}_{t \in [0,1]} L^{q'}(\R^d)},
\]
for all $F \in \mathcal{S}(\R^{d+1}).$ As we have that $\frac{1}{r} = \frac{1}{2} - \frac{1}{q},$ the value of $s = (d-3) \left(\frac{1}{2} - \frac{1}{q}\right)$ allows us to use the Hardy--Sobolev inequality: 
\begin{align*}
\left\| \int (-\Delta)^s U(t-t') F(\cdot,t') \, \mmd t' \right\|_{L^r_{t \in [0,1]} L^q(\R^d)} &\lesssim \left\| \int \|  (-\Delta)^s U(t-t') F(\cdot,t') \|_{L^q(\R^d)} \, \mmd t' \right\|_{L^r} \cr 
											  &\lesssim \left\| \int \frac{\|F(\cdot,t')\|_{L^{q'}(\R^d)}}{|t-t'|^{1-\frac{2}{q}}} \, \mmd t' \right\|_{L^r} \cr
											  &\lesssim \|F\|_{L^{r'}_{[0,1]} L^{q'}(\R^d)}.
\end{align*}
For the $\frac{2}{q} + \frac{2}{r} < 1$ case, let $\tilde{q}$ be such that $\frac{2}{\tilde{q}} + \frac{2}{r} = 1,$ employ the strategy above for such $\tilde{q}$ and then apply the sharp version of the 
Sobolev embedding theorem. This finishes the proof. 
\end{proof}

\subsection{Linear Estimates} We will need, besides the Strichartz estimates from the previous subsection, some \emph{linear estimates} to free solutions to \eqref{eq linear ZK}. We remark that most of them 
are just direct adaptations of the low-dimensional settings, and thus the brevity in their proofs. See the mentioned references throughout the text for more details. 

\begin{proposition}[Kato smoothing]\label{prop kato smoothing} Let $u_0 \in \mathcal{S}(\R^d)$ and $U(t)$ be as before. Then it holds that 
\[
\| \nabla U(t) u_0\|_{L^{\infty}_x L^2_{y,t}} \lesssim \|u_0\|_{L^2}.
\]
\end{proposition}

\begin{proof} This proof is basically a remake of the result for the KdV \cite{KPV1} and of the two- and three-dimensional cases \cite{Faminskii, RV2}. 

In fact, we perform the change of variables $\vartheta = \xi(\xi^2 + |\eta|^2) = h_{\eta}(\xi)$ in the Fourier definition of $U(t)u_0.$ This yields 
\[
U(t)u_0 = \mathcal{F}^{-1}_{\vartheta,\eta} (e^{i x (h_{\eta})^{-1}(\vartheta)} (h_{\eta}^{-1})'(\vartheta) \widehat{u_0}(h_{\eta}(\vartheta),\eta)) (y,t).
\]
Using now Plancherel and inverting the change of variables, we obtain 
\[
\|U(t)u_0\|_{L^2_{y,t}} = \|(h_{\eta}^{-1})'(\vartheta)  \widehat{u_0}(h_{\eta}(\vartheta),\eta)\|_{L^2_{\vartheta,\eta}} = \| |h_{\eta}'(\xi)|^{-1/2} \widehat{u_0}(\xi,\eta,\tau)\|_{L^2_{\xi,\eta}}.
\]
By noticing that $h_{\eta}'(\xi) = 3\xi^2 + |\eta|^2 \sim \xi^2 + \eta^2 \sim \mathcal{F} (-\Delta),$ by applying the expression above to 
$u_0 = \nabla v_0$ we get to 
\[
\| \nabla U(t) v_0\|_{L^2_{y,t}} \lesssim \| v_0\|_{L^2_{x,y}},
\]
and thus taking the supremum of the left-hand side in $x$ gives us the result. 
\end{proof}

\begin{proposition}[Maximal estimate]\label{prop maximal estimate} Let $d \ge 3.$ It holds that 
\[
\| U(t) u_0\|_{L^4_x L^{\infty}_{y,t}} \lesssim \|u_0\|_{H^s},
\]
where $s > \frac{d}{2} - \frac{1}{4}.$ 
\end{proposition}

\begin{proof} As previously mentioned, Proposition \ref{thm strichartz} implies directly that 
\[
\|U(t)u_0||_{L^4_{x,y,t}} \lesssim \| \langle \nabla \rangle^{\frac{d-3}{4}} u_0\|_{L^2}.
\]
Now, if we use Sobolev embedding on the $t-$variable and the fact that $(\partial_t)^r U(t) = (\partial_x \Delta)^r U(t)$ (which follows from the fact that the 
time-space Fourier support of $U(t)u_0$ is the surface $\{\tau = \xi(\xi^2 + |\eta|^2)\}$), followed by Sobolev embedding in the $y-$variable, we obtain 
\[
\| U(t)u_0\|_{L^4_x L^{\infty}_{y,t}} \lesssim \| (\partial_x \Delta)^{1/4^+}  \langle \nabla \rangle^{\frac{d-2}{2}^+} u_0\|_{L^2}.
\]
A simple computation on the Fourier side shows that the last expression is bounded by $\| \langle \nabla \rangle^s u_0\|_{L^2}$ whenever $s > \frac{d}{2} - \frac{1}{4},$ 
as desired. 
\end{proof}

With these estimates in hands, we prove some \emph{retarded estimates} which will be key to perform the fix-point argument. 

\begin{proposition}\label{prop ret group} Let $f \in \mathcal{S}(\R^{d+1}).$ Then it holds that 
\[
\left\| \nabla \int_0^t U(t-t')f(\cdot,t') \, \mmd t'\right\|_{L^{\infty}_T L^2_{x,y}} \lesssim \|f\|_{L^1_x L^2_{y,T}}.
\]
\end{proposition}

\begin{proof} We compose the dual version of Proposition \ref{prop kato smoothing} with the fact that $U(t)$ is an unitary group; this implies that
\[
\left\| \nabla \int_0^T U(t-t') f(\cdot, t')\, \mmd t' \right\|_{L^2_{x,y}} \lesssim \|f\|_{L^2_x L^2_{y,T}}.
\]
The proposition then follows if applied to $\tilde{f}(z,t') = 1_{[0,t]}(t') f(z,t')$ and taking the supremum on $t\in[0,T]$ of the left-hand side.
\end{proof}

\begin{proposition}\label{prop ret smooth} Let $f \in \mathcal{S}(\R^{d+1}).$ Then it holds that 
 \[
 \left\| \nabla^2 \int_0^t U(t-t') f(\cdot,t') \, \mmd t'\right\|_{L^{\infty}_x L^2_{y,T}} \lesssim \|f\|_{L^1_x L^2_{y,T}}.
 \]
\end{proposition}

\begin{proof} The proof of this proposition follows, essentially, the same lines of \cite[Proposition~3.6]{RV2}. 

In fact, we start out by writing $\nabla^2 \displaystyle\int_0^t U(t-t') f(\cdot,t') \, \mmd t'$ as the sum 

\[ 
\frac{1}{2} \nabla^2 \int_{\R} U(t-t')f(t')\, \text{sign}(t-t') \, \mmd t' + \frac{1}{2} \nabla^2 \int_{\R} U(t-t')f(t')\, \text{sign}(t') \, \mmd t'.
\]
Denote the first term by $\nabla^2 F(t).$ Taking a space-time Fourier transform shows that 
$\mathcal{F}_{t,x,y} F(\xi,\eta,\tau) = \widehat{\text{sign}}(\tau- \xi(\xi^2 + |\eta|^2)) \mathcal{F}_{t,x,y}(f)(\tau,\xi,\eta). $
By Plancherel, we have 
\[
\| \nabla^2 F \|_{L^2_{y,t}} = \| K(\tau,x,|\eta|) * \mathcal{F}_{y,t}(f(x,\cdot))(\eta,\tau)\|_{L^2_{\eta,\tau}},
\]
where $K(\tau,x,|\eta|) = \displaystyle\int_{\R} e^{i x \xi} \frac{\xi^2 + |\eta|^2}{\tau - \xi(\xi^2 + |\eta|^2)} \, \mmd \xi.$ Notice that one of the consequences of 
the proof of Proposition 3.6 in \cite{RV2} is that $\|K\|_{L^{\infty}(\R^{d+1})}$ is \emph{uniformly bounded} on $\tau,x$ and $\eta.$ Therefore, an application
of Young's convolution inequality together with another of Plancherel implies that 
\[
\| \nabla^2 F \|_{L^2_{y,t}} \lesssim \|f\|_{L^1_x L^2_{y,t}}.
\]
This proves the asserted bound for the first term. For the second one, we simply use a combination of Proposition \ref{prop kato smoothing} with its dual version 
to $\tilde{f}(z,t) = f(z,t)\, \text{sign}(t).$ An easy computation then shows that this finishes the proof.
\end{proof}

\begin{proposition}\label{prop ret max} Let $f \in \mathcal{S}(\R^{d+1}).$ Then it holds that 
\[
\left\| \int_0^t U(t-t') \Delta_k f(t') \, \mmd t'\right\|_{L^4_x L^{\infty}_{y,T}} \lesssim 2^{(s_d - 1)^+k}\|\Delta_k f\|_{L^1_x L^2_{y,T}},
\]
for all $k \ge 0,$ where we let $s_d = \frac{d}{2} - \frac{1}{4}.$ 
\end{proposition}

\begin{proof} Notice that, from Proposition \ref{prop maximal estimate} and the dual version of \ref{prop kato smoothing}, we get that
\[
\left\| \int_0^T U(t-t') \Delta_k f(t') \, \mmd t' \right\|_{L^4_x L^{\infty}_{y,T}} \lesssim 2^{(s_d-1)^+ k} \|\Delta_k f\|_{L^1_x L^2_{y,T}}.
\]
We now use an anitropic version of the Christ--Kiselev lemma \cite{CK}, which can be found in Theorem B.3, part (i), from \cite{BLC}. Indeed, our exponents meet well the conditions,
since $4 = \min(4,\infty) > 2 = \max\left(2, 1, \frac{2 \cdot 1}{2}\right).$ This implies the desired retarded estimate, as wished.
\end{proof}

\subsection{Proof of the well-posedness result} Taking the previously stated linear estimates as a starting point, we create a set of norms designed to take into consideration 
the smoothing and maximal estimates of \eqref{eq linear ZK}. 

In fact, like in \S \ref{sec 2d counter}, we define a set of auxiliary norms related to the 
estimates we have. Let then 
\[
\| u \|_{\mathcal{Y}^s_T(d)} = \|u\|_{L^{\infty}_T H^s_{x,y}} + \|\langle \nabla \rangle^{s-s_d^+} u\|_{L^4_x L^{\infty}_{y,T}} + \| \langle \nabla \rangle^{s+1} u\|_{L^{\infty}_x L^2_{y,T}}.
\]
We then define the norms 
\[
\|u\|_{\mathcal{X}^s_T(d)} = \left\| 2^{sj} \| \Delta_j u \|_{\mathcal{Y}^s_T} \right\|_{\ell^2(\N)},
\]
which defines the space in which we shall perform the iteration argument. Indeed, fix $k \ge 4$ and consider the Duhamel operator associated to \eqref{eq gZK}:
\begin{equation}\label{eq duhamel gZK} 
\Gamma_{u_0}(u)(t) = U(t)u_0 + \int_0^t U(t-t') \partial_x (u^{k+1})(t') \, \mmd t'.
\end{equation}
We seek to prove that $\Gamma_{u_0}$ preserves some metric space 
$$\mathcal{E}_{a;d}(T) = \{ v \in \mathcal{X}^s_T(d) \colon \|v\|_{\mathcal{X}^s_T(d)} \le a\},
$$ 
and is, in fact, a \emph{contraction} there. From Propositions \ref{prop kato smoothing} and \ref{prop maximal estimate}, we have that 
\[
\|U(t)u_0\|_{\mathcal{X}^s_T(d)} \lesssim \|u_0\|_s, \, \forall s \in \R.
\]
Thus, we are left with bounding the integral term. On the other hand, using Propositions \ref{prop ret group}, \ref{prop ret smooth} and \ref{prop ret max} together 
with the definitions of the $\mathcal{Y}^s_T(d)$ and $\mathcal{X}^s_T(d)$ norms, we see that 
\begin{equation}\label{eq bound nonlinear}
\left\| \int_0^t U(t-t') \partial_x (u^{k+1})(t') \, \mmd t' \right\|_{\mathcal{X}^s_T(d)} \lesssim \left\| 2^{sj} \|\Delta_j (u^{k+1})\|_{L^1_x L^2_{y,T}} \right\|_{\ell^2(\N)}.
\end{equation}
Our task is then to bound the right-hand side of \eqref{eq bound nonlinear}. Now we follow the approach by Ribaud and Vento \cite{RV1} in the two-dimensional setting. 

We first notice that interpolating the second and third terms in the definition of $\| \cdot\|_{\mathcal{Y}^s_T(d)}$ gives 
\begin{equation}\label{eq interpolation highD}
2^{\alpha j} \| \Delta_j u\|_{L^p_x L^q_{y,T}} \lesssim \|\Delta_j u\|_{\mathcal{X}^s_T(d)},
\end{equation}
where $p = \frac{4}{1-\theta}, \, q = \frac{2}{\theta}$ and $\alpha = (s + (1+s_d)\theta - s_d)^-.$ In particular, setting $\theta = \frac{s_d}{s_d + 1},$ we have 
\[
2^{sj} \|\Delta_j u\|_{L^{p_1^+}_x L^{q_1^-}_{y,T}} \lesssim \|\Delta_j u \|_{\mathcal{X}^s_T(d)},
\]
where $p_1 = 4(s_d + 1), \, q_1 = \frac{2(s_d + 1)}{s_d}.$ 
From this point, in analogy to \S \ref{sec 2d counter}, a paraproduct decomposition of $\Delta_j(u^{k+1})$ and H\"older's inequality shows that 
\[
\|\Delta_j(u^{k+1})\|_{L^1_x L^2_{y,T}} \lesssim \sum_{l \ge j-1} \| \Delta_l u\|_{L^{p_1^+}_x L^{q_1^-}_{y,T}} \| P_l u\|_{L^{(k p_1')^-}_x L^{ (k \tilde{q_1})^+}_{y,T}}^k.
\]
Here we use the notation $p_1'$ and $\tilde{q_1}$ for the positive reals so that $\frac{1}{p_1} + \frac{1}{p_1'} = 1$ and $\frac{1}{q_1} + \frac{1}{\tilde{q_1}} = \frac{1}{2}.$ 
The first term in the product of the right-hand side of the equation above admits adequate bounds by the previous considerations, so it remains to estimate the second one. 

Indeed, let $q = \frac{2}{\theta} = k\tilde{q_1} = 2k(s_d + 1) \iff \theta = \frac{1}{k(s_d+1)}$ in \eqref{eq interpolation highD}. This yields 
\begin{equation}\label{eq interpolation highD2}
2^{(s + 1/k - s_d)^- j} \| \Delta_j u\|_{L^{\left(\frac{1}{4} - \frac{1}{4k(s_d + 1)}\right)^{-1}}_x L^{2k(s_d+1)}_{y,T}} \lesssim \|\Delta_j u \|_{\mathcal{X}^s_T(d)},
\end{equation}
for all $j \ge 0.$ Notice now that 
\[
\| P_l u\|_{L^{(k p_1')^-}_x L^{ (k \tilde{q_1})^+}_{y,T}} \lesssim \sum_{r = 0}^l \|\Delta_r u\|_{L^{(k p_1')^-}_x L^{ (k \tilde{q_1})^+}_{y,T}} 
\]
\begin{equation}\label{eq many indices}
\lesssim \sum_{r = 0}^l 2^{(1/4 - 1/k)^+ r}\| \Delta_r u\|_{L^{\left(\left(\frac{1}{4} - \frac{1}{4k(s_d + 1)}\right)^{-1}\right)^+}_x L^{2k(s_d+1)^+}_{y,T}},
\end{equation}
by Sobolev embedding on the $x-$variable. Notice that we need $k \ge 4$ in order for this step to work. 

The right-hand side of \eqref{eq many indices} above is bounded by $T^{\delta} \| u\|_{\mathcal{X}^s_T(d)}$ as long as 
$s + \frac{1}{k} - s_d > \frac{1}{4}  - \frac{1}{k} \iff s > \frac{d}{2} - \frac{2}{k}.$ Indeed, this follows from interpolating \eqref{eq interpolation highD2} 
with the trivial bound 
\[
2^{-j^+} \| \Delta_j u\|_{L^N_{x,y,T}} \lesssim T^{\delta} \|\Delta_j u\|_{L^{\infty}_T L^2_{x,y}},
\]
where we take $N \gg 1$ sufficiently large. Collecting these bounds, it holds that 
\[
\left\| 2^{sj} \| \Delta_j (u^{k+1}) \|_{L^1_x L^2_{y,T}}\right\|_{\ell^2(\N)} \lesssim \| (1_{j \ge 0} 2^{-sj} ) *_j \| \Delta_j u\|_{\mathcal{X}^s_T(d)}\|_{\ell^2(\N)} \|u\|_{\mathcal{X}^s_T(d)}^k.
\]
By the discrete version of Young's convolution inequality, the latter term is controlled by $\|u\|_{\mathcal{X}^s_T(d)}.$ Therefore,
\[
\| \Gamma_{u_0} (u)\|_{\mathcal{X}^s_T(d)} \le C_s \|u_0\|_{H^s} + C\cdot T^{\delta} \|u\|_{\mathcal{X}^s_T(d)}^{k+1},
\]
and thus, for $a = 2C_s \|u_0\|_{H^s}$ and $T \sim_s (1+\|u_0\|_{H^s})^{-\beta_s},$ for some $\beta_s > 0,$ it holds that $\Gamma_{u_0}$ maps $\mathcal{E}_{a;d}(T)$ to itself. 
Moreover, redoing all the computations above with $\Gamma_{u_0} u - \Gamma_{u_0} v$ yields that it is, in fact, a \emph{contraction} on such space for such $T.$ Therefore, it has a unique fixed
point, which is our desired solution. 

By the standard-by-now methods, we conclude that this solution is unique and, by the fact that $\Gamma_{u_0}$ was Lipschitz, we conclude that the data-to-solution map is, in fact, locally Lipschitz 
on $C([0,T] \colon H^s) \cap \mathcal{X}^s_T(d),$ as desired. 

\subsection{Pointwise convergence of the flow} In this subsection, we discuss the proof of Theorem \ref{thm pointwise}. Indeed, we begin with an approximation lemma, which in turn is based off \cite[Proposition~3.3]{CLS}. In what follows, 
we denote by $u_N$ the (unique) solution to 
\begin{equation}\label{eq gZK smooth}
\begin{cases} 
\partial_t u_N + \partial_x \Delta u_N + \partial_x (P_N(u_N^{k+1})) = 0 & \text{ on } \R^d \times \R; \\
u_N(x,0) = P_N u_0(x), & \text{ on } \R^d. \\
\end{cases}
\end{equation}
By the energy method, for instance, we can see that solutions to \eqref{eq gZK smooth} are \emph{smooth}, as the initial data $P_N u_0$ is smooth for any $u_0 \in H^s.$ This fact will be crucial in the proofs below.

We start by proving that the $L^4-$maximal-in-time estimate for the group $U(t)$ fulfills our purposes whenever $d=2$ and $k \ge 2$, $d = 3, k\ge 3$ or $d \ge 3, k \ge 4.$ 

\begin{lemma}\label{lemma maximal} Suppose that, for $u_0 \in H^s(\R^d),$ we have that 
\[
\|u_N - u\|_{L^4_{x,y} L^{\infty}_T} \to 0 \text{ as } N \to \infty. 
\]
Then it holds that $u(z,t) \to u_0(z)$ as $t \to 0$ for almost every $z \in \R^d.$ 
\end{lemma}

\begin{proof} The proof is another instance of the relationship between maximal functions and pointwise convergence. Indeed, by smoothness of $u_N,$ it holds that 
\[
u_N(z,t) \to P_N u_0(z) \text{ for all } z \in \R^d.
\]
Therefore, 
\[
\limsup_{t \to 0} |u(z,t) - u_0(z)| \le \limsup_{t \to 0} |u(z,t) - u_N(z,t)| + |(I-P_N)u_0(z)|.
\]
Thus, by Chebyshev's inequality, 
\[
m(\{z \in \R^d \colon \limsup_{t \to 0} |u(z,t) - u_0(z)| > \eps\}) \le \frac{1}{\eps^4} \|u-u_N\|_{L^4_{x,y}L^{\infty}_T}^4 + \frac{1}{\eps^2}\|(I-P_N)u_0\|_{L^2}^2.
\]
Notice that the two terms in the right-hand side above can be made arbitrarily small by letting $N \to \infty.$ Thus, we conclude that 
$m(\{z \in \R^d \colon \limsup_{t \to 0} |u(z,t) - u_0(z)| > \eps\}) = 0, \, \forall \, \eps > 0.$ This concludes the proof. 
\end{proof}

\begin{proof}[Proof of Theorem \ref{thm pointwise}, $d \ge 2, k \ge d$ case] We start by writing the Duhamel formulation of both \eqref{eq gZK} and \eqref{eq gZK smooth}. It gives us that 
\begin{equation}\label{eq pointwise final}
\|u_N - u\|_{L^4_{x,y}L^{\infty}_T}\! \lesssim\! \|u_0 - P_N u_0\|_{H^{\tilde{s}_d^+} }\! + \!\Big\|\int_0^t U(t-t') \partial_x(P_N (u_N^{k+1}) - u^{k+1})(t') \, \mmd t' \Big\|_{L^4_{x,y}L^{\infty}_T},
\end{equation}
where we used Proposition \eqref{prop maximal estimate} in the first term. In order to control the integral term, we observe that in the same way we proved Proposition \ref{prop ret max}, it holds that, 
\[
\left\| \int_0^t U(t-t') \partial_x \Delta_j f(t') \, \mmd t' \right\|_{L^4_{x,y}L^{\infty}_T} \lesssim 2^{\tilde{s}_d^+ j} \| \Delta_j f\|_{L^1_x L^2_{y,T}}.
\]
Thus, for all $s > \tilde{s}_d,$ the integral term $\Big\| \displaystyle\int_0^t U(t-t') \partial_x(P_N(u_N^{k+1}) - u^{k+1}) (t') \, \mmd t' \Big\|_{L^4_{x,y} L^{\infty}_T}$ is controlled by 

\begin{equation}\label{eq triangle}
\begin{split}
 C_s \left\| 2^{sj} \| \Delta_j (P_N(u_N^{k+1}) - u^{k+1})\|_{L^1_x L^2_{y,T}} \right\|_{\ell^2_j} \lesssim_s & \left\| 2^{sj} \| \Delta_j (P_N(u_N^{k+1} - u^{k+1}))\|_{L^1_x L^2_{y,T}} \right\|_{\ell^2_j} \\
													      & + \left\| 2^{sj} \| \Delta_j((I-P_N)u^{k+1}\|_{L^1_x L^2_{y,T}} \right\|_{\ell^2_j}. 
\end{split}													      
\end{equation}
Now the first term on the right-hand side of \eqref{eq triangle} is controlled, by Young's inequality, by $\|u-u_N\|_{\mathcal{X}^s_T(d)},$ with the definition of the 
$\mathcal{X}^s_T(d)$ spaces we have adopted throughout the text, as $s > \max(\frac{d}{2} - \frac{2}{k}, \tilde{s}_d).$ On the other hand, it is easy to see from the monotone convergence theorem, together 
with the proofs of Theorem \ref{thm gZK dk}, Theorem \ref{thm reproof} and Theorem 1.1 in \cite{RV1}, that the second term on the right-hand side of \eqref{eq triangle} goes to $0$ as $N \to \infty.$ Inserting back into 
\eqref{eq pointwise final}, we obtain, for $N \gg 1,$ 
\[
\| u_N - u\|_{L^4_{x,y} L^{\infty}_T} \lesssim_s \delta + \|u_N - u\|_{\mathcal{X}^s_T(d)}. 
\]
On the other hand, by Theorem \ref{thm gZK dk}, for $N$ sufficiently large and $T = T(\|u_0\|_s),$ it holds that 
\[
\|u_N - u\|_{\mathcal{X}^s_T (d)} \lesssim_s \|(I-P_N)u_0\|_{H^s} \lesssim_s \delta.
\]
This promptly implies that $\|u_N - u\|_{L^4_{x,y} L^{\infty}_T}$ can be made arbitrarily small as $N \to \infty,$ given $s > \max\left(\frac{d}{2} - \frac{2}{k}, \tilde{s}_d\right).$ 
The conditions of Lemma \ref{lemma maximal} are then met, and we have concluded the proof.
\end{proof}

In order to handle the $d=3,\, k=2$ case, we remark that another maximal estimate with bounds \emph{independent} of the dimension holds in the case of the group $U(t).$ In fact, a result 
by Cowling \cite{Cowling} has as by-product that, whenever $\Omega: \R^d \to \R$ is smooth and homogeneous of degree $m,$ then 
\[
\left\| \sup_{t \in [0,1]} |e^{it\Omega(D)} f| \right\|_{L^2(\R^d)} \lesssim \|f\|_{H^s},
\]
whenever $s > m/2.$ In our case, we obtain that the global bound 
\[
\| U(t) f\|_{L^2(\R^d) L^{\infty}_{[0,1]}} \lesssim \|f\|_{H^s}
\]
holds whenever $s> \frac{3}{2}.$ By Roger's local-to-global transference principle (see Theorem \ref{theoo b}), we have that the \emph{local} maximal bound
\begin{equation}\label{eq maximal local}
\| U(t) f\|_{L^2(B^d(0,1)) L^{\infty}_{[0,1]}} \lesssim \|f\|_{H^s}
\end{equation}
holds for all $s > \frac{1}{2}$ and all dimensions $d \ge 2.$ We will use this local maximal bound in the three-dimensional case of the modified Zakharov--Kuznetsov equation. 

\begin{lemma}\label{lemma maximal 3} Let $u, u_N$ denote solutions to \eqref{eq gZK} and \eqref{eq gZK smooth}, respectively, in the case $k=2, d=3.$ Suppose that, for $u_0 \in H^s(\R^3),$ we have the existence
of $T = T(\|u_0\|_s)$ so that 
\[
\|u_N - u\|_{L^2(B^3(0,1)) L^{\infty}_T} \to 0 \text{ as } N \to \infty. 
\]
Then it holds that $u(z,t) \to u_0(z)$ as $t \to 0$ for almost every $z \in \R^3.$ 
\end{lemma}

\begin{proof} The proof of this lemma is almost identical to that of Lemma \ref{lemma maximal}, only that this time we employ the Chebyshev inequality argument 
on a fixed unit ball of $\R^3,$ use the translation invariance of the equations involved and cover the euclidean space $\R^3$ by countably many such balls. We omit the details.
\end{proof}

\begin{proof}[Proof of Theorem \ref{thm pointwise}, $d=3, k=2$ case] By \cite[Lemma~2.9]{Tao} (see also \cite[Lemma~2.1]{CLS} and the comments thereafter), the maximal
estimate \eqref{eq maximal local} implies the continuous embedding $X^{s,b}_{\delta} \hookrightarrow L^2(B^3(0,1)) L^{\infty}_{[0,\delta]}$ for the Bourgain space $X^{s,b}_{\delta}$ when $s,b > \frac{1}{2},$
where we define the norm 
\[
\|F\|_{X^{s,b}} = \| \langle (\xi,\eta) \rangle^s \langle \tau - \xi(\xi^2 + |\eta|^2)\rangle^b\, \widehat{u}(\xi,\eta,\tau) \|_{L^2_{\xi,\eta,\tau}},
\]
where we have used the space-time Fourier transform above. Thus, it holds that 
\[
\left\| \sup_{t \in [0,\delta]} |F(x,t)| \right\|_{L^2(B^3(0,1))} \lesssim \|F\|_{X^{s,b}_{\delta}}, \, \forall \, F  \in X^{s,b}_{\delta}.
\]
This readily implies that 
\[
\|u-u_N\|_{L^2_{B^3} L^{\infty}_T} \lesssim \|u-u_N\|_{X^{s,b}_T}.
\]
By the Duhamel principle applied to $u,u_N$ and the properties of the Bourgain spaces $X^{s,b}_T,$ we see that 
\begin{equation}\label{eq pointwise bourgain}
\begin{split}
\|&u-u_N\|_{X^{s,b}_T} \\
&\lesssim \|(I-P_N)u_0\|_{H^s(\R^3)} + \|P_N \partial_x(u^3 - (u_N)^3))\|_{X^{s,b'}_T} + \|(I-P_N) \partial_x (u^3)\|_{X^{s,b'}_T},
\end{split}
\end{equation}
for some $b' > -\frac{1}{2}.$ Gr\"unrock's trilinear estimate \cite[Proposition~1]{Gru1} for the modified Zakharov--Kuznetsov equation in dimension three implies then that, for each $s>\frac{1}{2},$ 
there must be $b' > -\frac{1}{2}$ so that for all $b > \frac{1}{2},$
\begin{equation}\label{eq trilinear three}
\| \partial_x(u^3 - v^3)\|_{X^{s,b'}_T} \lesssim T^{\delta} \left(\|u\|_{X^{s,b}_T}^2 + \|v\|_{X^{s,b}_T}^2\right) \|u-v\|_{X^{s,b}_T}.
\end{equation}
The first consequence of \eqref{eq trilinear three} is by setting $v = u_N,$ which shows that 
\[
\|P_N \partial_x(u^3 - (u_N)^3))\|_{X^{s,b'}_T} \lesssim T^{\delta} C(\|u_0\|_s) \|u-u_N\|_{X^{s,b}_T}.
\]
By taking $T$ sufficiently small, it holds that the right-hand side of the expression above can be absorbed into the left-hand side 
of \eqref{eq pointwise bourgain}. As a second consequence, 
setting $v \equiv 0$ on \eqref{eq trilinear three} and using the definition of $X^{s,b}_T,$ we have that 
\[
\| (I-P_N) \partial_x (u^3)\|_{X^{s,b}_T} \to 0 \;\text{ as }\; N \to \infty. 
\]
Therefore, as $u_0 \in H^s, \, s> \frac{1}{2},$ we see that the left-hand side of \eqref{eq pointwise bourgain} converges to 0 as $N \to \infty.$ This finishes this case by Lemma \ref{lemma maximal 3}, and thus also the proof of Theorem 
\ref{thm pointwise}. \end{proof}

\section{Comments and Remarks} 

\subsection{Sharp maximal estimates in two dimensions} As discussed in \S \ref{sec sharp}, the space-time maximal estimates \eqref{eq maximal space time} are, in fact, \emph{sharp} for the three-dimensional case, 
due to our counterexamples and the previous works \cite{RV2, Gru1}. 

On the other hand, although we have made progress in the question of sharpness of space-time maximal estimates in dimension 2 through simple counterexamples as well as through 
the indirect method given in \eqref{sec 2d counter}, we still have a gap where next to nothing is known about sharpness. Indeed, for $p \in (2,4),$ we only know that the estimate
\begin{equation}\label{eq sharp 2d}
\| U(t)u_0\|_{L^p_x L^{\infty}_{y,T}} \lesssim \|u_0\|_s 
\end{equation}
holds, by interpolation, in the $s > \frac{3}{4}$ range. It is likely, however, that this is \emph{not} the sharp range for those indices; in fact, we conjecture the following: 
\begin{conjecture}\label{conj sharp} If $d=2$, $p=3,$ then \eqref{eq sharp 2d} holds for all $s > \frac{2}{3}.$ 
\end{conjecture}
A hand-waving justification for such a conjecture is the following: by the results in \cite{Gru2}, local well-posedness for the quartic Zakharov--Kuznetsov equation in two dimensions holds 
in $H^s,$ whenever $s > \frac{1}{3}.$ On the other hand, considering the modified norms 
\[
\|u\|_{\tilde{Y}^s_T} = \|u\|_{L^{\infty}_T L^2_{x,y}} + \|\langle \nabla \rangle^{s - \frac{2}{3}^+} u\|_{L^3_x L^{\infty}_{y,T}} + \|\langle \nabla \rangle^{s+1} u\|_{L^{\infty}_x L^2_{y,T}},
\]
with $\|u\|_{\tilde{X}^s_T} = \left\| 2^{sj} \|\Delta_j u\|_{\tilde{Y}^s_T}\right\|_{\ell^2(\N)},$ and reproducing the argument in \S \ref{sec 2d counter}, we see that if \eqref{eq sharp 2d} holds for 
all $s> \frac{2}{3},$ we recover the full range $s > \frac{1}{3}$ of local well-posedness. 

For other values of $p \in (2,4)$, it is not crystal clear what should happen. In fact, when $p \in (3,4),$ Conjecture \ref{conj sharp} plus Proposition \ref{thm necessary} would imply
that those estimates are, in fact, \emph{sharp} for the interpolation between the $L^4_xL^{\infty}_{y,T}$ and the $L^3_xL^{\infty}_{y,T}.$ On the range $p \in (2,3),$ on the other hand, 
we do not know what to expect: on the one hand, the range should \lq\lq blow up" to $s > \frac{3}{4}$ as $p \to 2;$ on the other hand, it is not, in principle, impossible for the range to be the 
necessary one given by Proposition \ref{thm necessary} up until the $p=2$ endpoint. 

\subsection{Maximal estimates and LWP for the $k=3$ case} As previously remarked, the Strichartz estimate
\[
\|U(t)u_0\|_{L^4_{x,y,t}} \lesssim \|u_0\|_{H^s}, \; s > \frac{d-3}{4},
\]
only allows us to prove local well-posedness for the IVP \ref{eq gZK} if $k \ge 4,$ and provides us with the full subcritical range of results in such cases by 
passing to suitable maximal functions. Together with the $k=1$ case in \cite{HK} and $k=2$ in \cite{Kinoshita2}, the only remaining case for proving local well-posedness for the 
generalized Zakharov--Kuznetsov equation in higher dimensions is $k=3.$ For $d=2,3,$ such a result was obtained by Gr\"unrock by using 
a suitable modification of the Ribaud--Vento techniques. 

One of the main features of Gr\"unrock's proof is the use of \emph{Strichartz estimates with derivative gain.} In fact, he is able to reach the full subcritical range by using the estimates 
\[
\| K(D_x,D_y)^{1/8} U(t) u_0 \|_{L^4_{x,y,t}(\R^3)} \lesssim \|u_0\|_{L^2(\R^2)},
\]
where we define $K(D_x,D_y)^{\sigma}g = \mathcal{F}^{-1}_{x,y} |3\xi^2 - \eta^2|^{\sigma} \mathcal{F}_{x,y}g,$ and 
\[
\| D_x^{1/10} U(t)u_0\|_{L^{\frac{15}{4}}(\R^4)} \lesssim \|u_0\|_{L^2(\R^3)}.
\]
Unfortunately, such estimates seem not to be available in higher dimensions, with an exception for the Kato Smoothing estimate in Proposition \ref{prop kato smoothing}.  
Instead of these, one may use the \emph{maximal estimates}, as illustrated by our method, in order to reach the sharp bounds. Indeed, as the space-time maximal estimate \eqref{eq maximal space time} 
is sharp in three dimensions, the argument in Section 3 above can be refined to prove local well-posedness for the whole subcritical range $s > \frac{3}{2} - \frac{2}{k}$ 
in three dimensions for \emph{all} $k \ge 2.$ 

As a matter of fact, we only need to redefine the norms defining the spaces $\mathcal{X}^s_T(3)$ as follows: let first 
\[
\|u\|_{\tilde{\mathcal{Y}}^s_T(3)} = \|u\|_{L^{\infty}_T L^2_{y,T}} + \| \langle \nabla_{x,y} \rangle^{s-1^+} u\|_{L^2_x L^{\infty}_{y,T}} + \| \langle \nabla_{x,y} \rangle^{s+1} u\|_{L^{\infty}_x L^2_{y,T}},
\]
and then $\|u\|_{\tilde{\mathcal{X}}^s_T(3)} = \left\| 2^{sj} \|\Delta_j u \|_{\tilde{\mathcal{Y}}^s_T(3)} \right\|_{\ell^2(\N)}.$ The computations 
performed in the proof of Theorem \ref{thm gZK dk} adapt easily due to the sharp bound 
\[
\|U(t)u_0\|_{L^2_x L^{\infty}_{y,T}} \lesssim \|u_0\|_s, \, s > 1,
\]
and we reobtain the following result:
\begin{theorem}[Main result in \cite{Gru1}; Three-dimensional result in \cite{Gru2}]\label{thm reproof} Let $d = 3$ and $k \in \{2,3\}.$ Then there are function spaces $\tilde{\mathcal{X}}^s_T$ so that for each $u_0 \in H^s(\R^3)$ with $s > \frac{3}{2} - \frac{2}{k},$ 
the IVP \eqref{eq gZK} has a unique solution 
\[
u \in C([0,T] \colon H^s) \cap \tilde{\mathcal{X}}^s_T,
\]
where $T = T(\|u_0\|_s) > 0.$ Moreover, the map $u_0 \mapsto u$ from $H^s(\R^3)$ to $\tilde{\mathcal{X}}^s_T \cap C([0,T] \colon H^s)$ is locally Lipschitz continuous. 
\end{theorem}

As mentioned previously, we (still) do not know whether the conjectured sharp bound 
\begin{equation}\label{eq l2 alld}
\| U(t) u_0\|_{L^2_x L^{\infty}_{y,T}} \lesssim \|u_0\|_s, \; s > \frac{d-1}{2}
\end{equation}
holds in order to adapt the idea given above to the high-dimensional setting. Nevertheless, the Sobolev embedding theorem easily implies that 
\[
\| U(t)u_0\|_{L^2_x L^{\infty}_{y,T}} \lesssim \| \langle \nabla \rangle^{\frac{d-1}{2}^+} U(t) u_0 \|_{L^2_{x,y} L^{\infty}_T} \lesssim \| (\partial_x \Delta)^{\frac{1}{2}^+} \langle \nabla \rangle^{\frac{d-1}{2}^+} U(t) u_0 \|_{L^2_{x,y,T}},
\]
which, on the other hand, is bounded by $T^{1/2} \|u_0\|_{H^s},$ whenever $s > \frac{d+2}{2}.$ By defining the norm 
\[
\|u\|_{\tilde{\mathcal{Y}}^s_T(d)} = \|u\|_{L^{\infty}_T L^2_{y,T}} + \| \langle \nabla_{x,y} \rangle^{s-\left(\frac{d+2}{2}\right)^+} u\|_{L^2_x L^{\infty}_{y,T}} + \| \langle \nabla_{x,y} \rangle^{s+1} u\|_{L^{\infty}_x L^2_{y,T}}
\]
and consequently $\|u\|_{\tilde{\mathcal{X}}^s_T(d)} = \left\| 2^{sj} \|\Delta_j u\|_{\tilde{\mathcal{Y}^s_T(d)}} \right\|_{\ell^2(\N)},$ the aforementioned techniques to 
set the Picard iteration scheme in motion imply immediately the following result:

\begin{theorem} For $d \ge 4,\, k =3,$  we have that for each $u_0 \in H^s(\R^d)$ with $s > \frac{d}{2} + \frac{5}{6},$ 
the IVP \eqref{eq gZK} has a unique solution 
\[
u \in C([0,T] \colon H^s) \cap \tilde{\mathcal{X}}^s_T,
\]
where $T = T(\|u_0\|_s) > 0.$ Moreover, the map $u_0 \mapsto u$ from $H^s(\R^3)$ to $\tilde{\mathcal{X}}^s_T \cap C([0,T] \colon H^s)$ is locally Lipschitz continuous. 
\end{theorem}

Currently, we believe that the estimate \eqref{eq l2 alld} holds for all $d \ge 3.$ Nonetheless, we also believe that the use of \eqref{eq l2 alld} is not strictly necessary 
in order to prove local well-posedness in the full subcritical range $s > \frac{d}{2} - \frac{2}{k}.$ Indeed, Kinoshita's method for proving local well-posedness (and, in fact, also 
global well-posedness on the critical Sobolev space for small data) only uses a bilinear estimate and Strichartz estimates for the group $U(t);$ therefore, it might be possible to 
reach the full range by employing Proposition \ref{thm strichartz} and other non-endpoint inequalities. 

\section*{Acknowledgements}
F. L. was partially supported by CNPq and FAPERJ, Brazil. J.P.G.R. acknowledges financial support from CNPq, Brazil.

\end{document}